\documentclass[11 pt]{amsart}

\usepackage{amsfonts}
\usepackage{amssymb}
\usepackage{amsmath}
\usepackage{latexsym}
\usepackage{mathrsfs}
\usepackage{amsthm}
\usepackage{verbatim}
\usepackage[all]{xy}

\newtheorem{thrm}{Theorem}[section]
\newtheorem{lem}[thrm]{Lemma}
\newtheorem{cor}[thrm]{Corollary}
\newtheorem{prop}[thrm]{Proposition}
\newtheorem{conj}[thrm]{Conjecture}
\newtheorem{thrmconj}[thrm]{Theorem-Conjecture}

\theoremstyle{definition}
\newtheorem{defn}[thrm]{Definition}

\newtheorem{exmple}[thrm]{Example}
\newtheorem{rmk}[thrm]{Remark}
\newtheorem{ques}[thrm]{Question}

\DeclareMathOperator{\Eff}{\overline{Eff}}

\DeclareMathOperator{\Nef}{Nef}

\DeclareMathOperator{\Mov}{Mov}

\DeclareMathOperator{\Chow}{Chow}

\DeclareMathOperator{\Supp}{Supp}

\DeclareMathOperator{\mc}{mc}

\DeclareMathOperator{\chdim}{chdim}

\newcommand{\kapclas}{\kappa_{\mathrm{classical}}}

\begin{document}

\title{Iitaka dimension for cycles}
\author{Brian Lehmann}
\thanks{The author was partially supported by an NSA Young Investigator Grant and by NSF grant 1600875.}
\address{Department of Mathematics, Boston College  \\
Chestnut Hill, MA \, \, 02467}
\email{lehmannb@bc.edu}

\begin{abstract}
We define the Iitaka dimension of a numerical cycle class and develop its theory.  We conjecture that the Iitaka dimension is integer-valued, and give some evidence in this direction.  We focus on two cases of geometric interest: Schubert cycles on Grassmannians and cycles contracted by morphisms.
\end{abstract}

\maketitle

\section{Introduction}

Let $X$ be a projective variety over $\mathbb{C}$ and let $N_{k}(X)_{\mathbb{Z}}$ denote the group of $k$-cycle classes up to numerical equivalence.  Given a class $\alpha \in N_{k}(X)_{\mathbb{Z}}$, we define the mobility count of $\alpha$ to be
\begin{equation*}
\mc(\alpha) = \max \left\{ b \in \mathbb{Z}_{\geq 0} \, \left| \, \begin{array}{c} \textrm{any }b\textrm{ general points of } X \textrm{ are contained}  \\ \textrm{in an effective cycle of class } \alpha  \end{array} \right. \right\}.
\end{equation*}
The mobility count is analogous to the dimension of the space of sections of a divisor: for a divisor $L$ the number of general points that can be imposed on members of $|L|$ is $h^{0}(X,L)-1$.
This analogy is richer than might be expected at first sight.   \cite{lehmann16} and \cite{fl13} show that one can understand the ``positivity'' of a cycle class $\alpha$ by studying the asymptotic behavior of $\mc(m\alpha)$ as $m$ increases.

In this paper, we study the asymptotic behavior of classes on the boundary of the pseudo-effective cone.  Continuing the analogy, we define:

\begin{defn} \label{iitakadimdef}
Let $X$ be a projective variety of dimension $n$ and let $\alpha \in N_{k}(X)_{\mathbb{Z}}$.  If some positive multiple of $\alpha$ is represented by an effective cycle, we define the Iitaka dimension of $\alpha$ to be
\begin{equation*}
\kappa(\alpha) := (n-k) \sup \left\{ r \in \mathbb{R}_{\geq 0} \left| \, \limsup_{m \to \infty} \frac{\mc(m\alpha)}{m^{r}} > 0 \right. \right\}.
\end{equation*}
Otherwise, we set $\kappa(\alpha) = -\infty$.
\end{defn}

Here the term $(n-k)$ is simply a convenient rescaling factor.    
It is not hard to show that the Iitaka dimension takes values in the set
\begin{equation*}
\kappa(\alpha) \in \{ -\infty \} \cup \{ 0 \} \cup [n-k,n].
\end{equation*}
Our main goal is to analyze the possible values of the Iitaka dimension and to understand their relationship with geometry.  
Our results are motivated by the following conjecture:

\begin{conj} \label{mainconj}
Let $X$ be a projective variety and let $\alpha \in N_{k}(X)_{\mathbb{Z}}$.  Then  $\kappa(\alpha) \in \mathbb{Z}_{\geq 0} \cup \{ - \infty \}$.
\end{conj}

This conjecture is perhaps surprising: while higher codimension cycles exhibit many pathologies not present for divisors, the conjecture predicts a cleaner picture from the viewpoint of positivity.    It is also worthwhile to study weaker variants.  For example, if we fix the dimensions $n,k$, are there only finitely many possible values of the Iitaka dimension?

If the Iitaka dimension is integer-valued, then it captures fundamental geometric information about $\alpha$, and it would be interesting to clarify this geometric input.  The following question is related to a conjecture of \cite{voisin10}.

\begin{ques}
Let $X$ be a smooth projective variety and let $\alpha \in N_{k}(X)_{\mathbb{Z}}$.  As we let $W$ vary over all effective cycles with class proportional to $\alpha$ which contain and are smooth at a fixed very general point $p$, is $\kappa(\alpha)$ determined by the set of tangent planes $T_{p}W \subset T_{p}X$?
\end{ques}

\begin{exmple}
Let $X$ be a smooth projective variety and $D$ be a Cartier divisor on $X$.  The usual Iitaka dimension of $D$ can only take integer values (see \cite{iitaka70}, \cite{iitaka71}).  Section \ref{divisorsec} shows that the result remains true in the numerical setting: for a Weil divisor class $\alpha$ on any projective variety $X$, $\kappa(\alpha) \in \{ - \infty, 0, 1, \ldots, \dim X \}$.
\end{exmple}

\begin{exmple}
Let $X$ be a projective variety of dimension $n$ and let $\alpha \in N_{k}(X)_{\mathbb{Q}}$ be any class.  \cite{lehmann16} shows that $\kappa(\alpha)$ attains its maximum value $n$ if and only if $\alpha$ lies in the interior of the pseudo-effective cone (in which case we say that $\alpha$ is big).  In fact more is true: there is a constant $\epsilon_{n,k} > 0$ such that $\kappa(\alpha)$ can not take values in the set $(n-\epsilon_{n,k},n)$, showing ``discreteness'' of the Iitaka dimension in a small neighborhood of $n$.
\end{exmple}

\begin{exmple} \label{mobilityofcurves}
Let $\alpha$ be a curve class on a projective variety $X$ of dimension $n$.  \cite[Theorem 2.4]{8authors} shows that if two general points of $X$ can be connected by an effective cycle with class proportional to $\alpha$, then $\alpha$ is big.  Thus there are four distinct behaviors for the Iitaka dimension of $\alpha$:
\begin{enumerate}
\item $\kappa(\alpha) = -\infty$.  By definition this happens when no positive multiple of $\alpha$ is represented by an effective cycle.
\item $\kappa(\alpha) = 0$.  This occurs when no positive multiple of $\alpha$ is represented by a curve through a very general point of $X$.  In particular $\mc(m\alpha) = 0$ for every $m>0$.
\item $\kappa(\alpha) = n-1$.  This occurs when there is an effective cycle of class proportional to $\alpha$ through one general point of $X$, but two general points can not be connected by a chain of such cycles.  Then the quotient theory of \cite{campana81} and \cite{kmm92} yields a rational map $g: X \dashrightarrow Z$ with $\dim Z > 0$ contracting all such curves through very general points.  In particular, there must be a positive constant $C$ such that $\mc(m\alpha) = Cm$ for every sufficiently divisible $m$.
\item $\kappa(\alpha) = n$.  By \cite{8authors}, the only other possibility is that $\alpha$ is big and that $\mc(m\alpha)$ has the maximal possible growth rate.
\end{enumerate}
\end{exmple}

We first prove some general results in support of Conjecture \ref{mainconj}, for example:

\begin{prop}
Let $X$ be a smooth projective variety of dimension $n \geq 3$.  Suppose that $\alpha \in \Eff_{n-2}(X)$ is an extremal ray and that there is an ample divisor $i: A \hookrightarrow X$ such that $\alpha \not \in A \cdot N^{1}(X)$.  Then $\kappa(\alpha) \leq n-1$.
\end{prop}

We then focus on two specific situations: classes contracted by morphisms, and Schubert classes on Grassmannians.  These examples can be seen as prototypes of arbitrary boundary classes, and so are particularly interesting as indicators of what to expect in general.

\subsection{Grassmannians}

Suppose that $X = G(m,n)$ is a Grassmannian of $m$ planes in an $n$-dimensional vector space and $\alpha$ is a Schubert class on $X$.  Given a non-increasing tuple of integers $\lambda = (\lambda_{1},\ldots,\lambda_{m})$ whose components $\lambda_{i}$ satisfy $0 \leq \lambda_{i} \leq n-m$, we let $\sigma_{\lambda}$ denote the class of the Schubert variety parametrizing linear subspaces $W$ whose dimension of intersection with the members of a fixed full flag $V_{\bullet}$ are determined by
\begin{equation*}
\dim(W \cap V_{n-m+i-\lambda_{i}}) \geq i.
\end{equation*}

We focus on the easiest case $G(2,n)$, where we can give a complete description of the Iitaka dimension for Schubert classes.

\begin{thrm} \label{iitakadimg2nintro}
The Iitaka dimension of a Schubert cycle on $G(2,n)$ is determined by the following list:
\begin{itemize}
\item $\kappa(\sigma_{1}) = \kappa(\sigma_{n-2,n-3}) = 2(n-2)$.
\item $\kappa(\sigma_{r}) = n-2$ for $1 < r \leq n-2$.
\item $\kappa(\sigma_{r,r-1}) = 2r$ for $1 < r < n-2$.
\item $\kappa(\sigma_{r,s}) = r+s$ otherwise.
\end{itemize}
\end{thrm}

In particular, the Iitaka dimension is usually the smallest possible value.


\begin{rmk}
The study of the Iitaka dimension for Schubert classes is closely related to the differential-geometric notion of Schubert rigidity developed in the series of papers \cite{walters97}, \cite{bryant05}, \cite{hong05}, \cite{hong07}, \cite{coskun11}, \cite{rt12}, \cite{robles13}, \cite{cr13}, \cite{coskun14}.

A Schubert class $\sigma$ is called multi rigid if the only effective cycles with class proportional to $\sigma$ are sums of Schubert varieties.  Note that any multi rigid class automatically has the minimal possible Iitaka dimension $\kappa(\sigma) = \mathrm{codim}(\sigma)$.  By comparison, the calculation of the Iitaka dimension yields a weaker conclusion but provides interesting geometric information for every class.
\end{rmk}

\subsection{Contracted classes}

We next turn to classes contracted by morphisms.  We say that $\alpha \in \Eff_{k}(X)$ is a $\pi$-contracted class if $\pi: X \to Z$ is a morphism such that $\pi_{*}\alpha = 0$.  If $\dim(Z) \leq k$ then any contracted class must lie on the boundary of the pseudo-effective cone.  Using the geometry of $\pi$ we can expect to obtain bounds on the Iitaka dimension of $\alpha$.  

\begin{exmple} \label{p2timesp2one}
Let $X = \mathbb{P}^{2} \times \mathbb{P}^{2}$.  Let $A$ and $H$ denote the pullback of the hyperplane class from the first and second factors respectively.  Then $\Eff_{2}(X)$ is a simplicial cone with generators $H^{2}, H \cdot A, A^{2}$.  The Iitaka dimensions of the non-zero boundary classes are determined by the extremal face:
\begin{enumerate}

\item $\alpha \in \mathbb{R}_{\geq 0} H^{2} + \mathbb{R}_{\geq 0} A^{2}$.  Each component of a cycle representing $m\alpha$ is a fiber of one of the projection maps and can contain at most one general point.  Thus $\kappa(\alpha) = 2$.  

\item $\alpha \in \mathbb{R}_{>0} (H \cdot A)$.  An irreducible cycle representing $m(A \cdot H)$ maps to a curve under each projection.  If the degrees of these curves are $d_{1}$ and $d_{2}$ the cycle can go through at most 
\begin{equation*}
\min \left\{ \left( \begin{array}{c} d_{1} + 2 \\ 2 \end{array} \right) - 1, \left( \begin{array}{c} d_{1} + 2 \\ 2 \end{array} \right) - 1 \right\} \approx \frac{1}{2} \min \{ d_{1}^{2} ,d_{2}^{2} \}
\end{equation*}
general points.  Since $d_{1}d_{2} = m$, the maximum possible bound occurs when $d_{1} \approx d_{2} \approx m^{1/2}$.  It is then not hard to show that $\kappa(\alpha) = 2$. 

\item $\alpha \in \mathbb{R}_{>0} H^{2} + \mathbb{R}_{>0} (H \cdot A)$ or its involutive face.  Theorem \ref{contractedthrm} shows that $\kappa(\alpha) = 3$.  We describe how to construct cycles achieving this growth rate; showing that this rate is the optimal one is somewhat harder.

Since any two classes in this face have comparable mobility count growth rates, it suffices to construct a single example $\alpha$ with $\kappa(\alpha) \geq 3$.  Let $\phi: \tilde{X} \to X$ be the blow-up along a fiber of the first projection map.  This variety admits a map $g: \tilde{X} \to \mathbb{P}^{2} \times \mathbb{P}^{1}$.  Let $\beta$ be a complete intersection curve on $\mathbb{P}^{2} \times \mathbb{P}^{1}$, so that $\mc(m\beta) \sim Cm^{3/2}$.  Then $g^{*}\beta$ is a surface class whose mobility count achieves the same growth rate, and its pushforward $\alpha = \phi_{*}g^{*}\beta$ has $\kappa(\alpha) \geq 3$.  Furthermore $\alpha$ lies in the interior of the desired face.
\end{enumerate}
\end{exmple}

The most useful framework for discussing contracted classes in general was set up by \cite{fl15} (see also \cite{cc15}).

\begin{defn}
Let $\pi: X \to Z$ be a surjective morphism of projective varieties and let $\alpha \in \Eff_{k}(X)$.  Fix an ample divisor $A$ on $Z$.  The $\pi$-contractibility index of $\alpha$ is defined to be the largest non-negative integer $c \leq k$ such that $\alpha \cdot \pi^{*}A^{k-c+1}=0$.  This definition is independent of the choice of $A$.
\end{defn}

The expected behavior of the Iitaka dimension for a contracted cycle $\alpha$ depends on the contractibility index.

\begin{thrmconj} \label{contractedconj}
Let $X$ be a projective variety of dimension $n$.  Suppose that $\pi: X \to Z$ is a surjective morphism of projective varieties of relative dimension $e$ and suppose $\alpha \in \Eff_{k}(X)_{\mathbb{Z}}$ has $\pi$-contractibility index $c$.   Then:
\begin{description}
\item[Theorem] If $c > e$, then $\kappa(\alpha) \leq 0$.
\item[Theorem] If $c = e$, then $\kappa(\alpha) \leq \dim $.
\item[Conjecture] If $k-\dim Z < c < e$, then $\kappa(\alpha) \leq n-c$.
\end{description}
\end{thrmconj}

The transition in behavior from $c > e$ to $c \leq e$ has the following geometric explanation.  \cite{fl15} defines a pseudo-effective class $\alpha \in N_{k}(X)$ to be ``$\pi$-exceptional'' if the $\pi$-contractibility index for $\alpha$ is larger than the relative dimension of $\pi$.  
It then shows that $\pi$-exceptional classes are ``rigid'' in a strong sense, and in particular can not contain a general point of $X$.   

We prove two statements in the direction of Conjecture \ref{contractedconj} for $c<e$.  First, we show that the conjectural upper bound cannot be improved: for any morphism $\pi: X \to $, there exists a class $\alpha \in \Eff_{k}(X)$ of contractibility index $c$ achieving the stated bound on the Iitaka dimension.  Second, we prove Conjecture \ref{contractedconj} when $k-c$ is at most $1$.  In particular:

\begin{thrm}
Conjecture \ref{contractedconj} holds if either
\begin{itemize}
\item $X$ has dimension $\leq 4$, or
\item $k \leq 2$.
\end{itemize}
\end{thrm}

\begin{exmple} \label{p2timesp2two}
Consider again $X = \mathbb{P}^{2} \times \mathbb{P}^{2}$ equipped with the first projection map $\pi: X \to \mathbb{P}^{2}$.  The contractibility index of a non-zero pseudo-effective class $a H^{2} + b (H \cdot A) + c A^{2}$ is simply the smallest exponent of $A$ appearing in a term with non-zero coefficient.  Then Theorem \ref{contractedconj} is verified explicitly by Example \ref{p2timesp2one}.
\end{exmple}

\subsection{Acknowledgements}

I would like to thank I.~Coskun for a helpful conversation about Grassmannians.

\section{Background}

Throughout we work over $\mathbb{C}$.  Varieties are irreducible and reduced.  A cycle will always mean a $\mathbb{Z}$-cycle unless otherwise qualified, and a numerical class will always mean an $\mathbb{R}$-class unless otherwise qualified.

\subsection{Numerical spaces and cones}

For a projective variety $X$, we let $N_{k}(X)_{\mathbb{Z}}$ denote the abelian group of $k$-cycles up to numerical equivalence as in \cite{fulton84}.  We then set
\begin{align*}
N_{k}(X)_{\mathbb{Q}} & := N_{k}(X)_{\mathbb{Z}} \otimes_{\mathbb{Z}} \mathbb{Q} \\
N_{k}(X) & := N_{k}(X)_{\mathbb{Z}} \otimes_{\mathbb{Z}} \mathbb{R} 
\end{align*}
We let $N^{k}(X)$ denote the dual space of $N_{k}(X)$ consisting of $\mathbb{R}$-polynomials in Chern classes of vector bundles on $X$ up to numerical equivalence (and similarly define the dual groups $N^{k}(X)_{\mathbb{Q}}$ and $N^{k}(X)_{\mathbb{Z}}$).  There is an intersection product $N^{\ell}(X) \times N_{k}(X) \to N_{k-\ell}(X)$.  We refer to \cite{fl14} for a discussion of these spaces and their behavior under morphisms.

The pseudo-effective cone $\Eff_{k}(X) \subset N_{k}(X)$ is the closure of the cone generated by classes of effective cycles on $X$.  It is a full-dimensional proper convex closed cone.  A class in the interior of $\Eff_{k}(X)$ is called big.  We will use the notation $\alpha \preceq \beta$ to denote that $\beta - \alpha \in \Eff_{k}(X)$.  The dual cone to $\Eff_{k}(X)$ is the nef cone and is denoted $\Nef^{k}(X)$.

\begin{lem}  \label{intprop}
Let $X$ be a projective variety and let $H$ be an ample divisor.
\begin{enumerate}
\item Suppose that some multiple of $\alpha \in \Eff_{k}(X)$ is represented by an effective class.  Then some multiple of $H \cdot \alpha$ is represented by an effective class.  In particular, $H \cdot$ preserves pseudo-effectiveness.
\item If $\alpha \in \Eff_{k}(X)$ is big, then $H \cdot \alpha \in \Eff_{k-1}(X)$ is big.
\end{enumerate}
\end{lem}

\subsection{Families of cycles and mobility count} \label{familysec}

For us, the most convenient definition of a family of cycles is the following.

\begin{defn}  \label{familydef}
Let $X$ be a projective variety.  A family of $k$-cycles on $X$ consists of a variety $W$, a reduced closed subscheme $U \subset W \times X$, and an integer $a_{i}$ for each component $U_{i}$ of $U$, such that for each component $U_{i}$ of $U$ the first projection map $p: U_{i} \to W$ is flat dominant of relative dimension $k$.  If each $a_{i} \geq 0$ we say that we have a family of effective cycles.  We say that $\sum a_{i}U_{i}$ is the cycle underlying the family.  
\end{defn}

We will usually denote a family of $k$-cycles using the notation $p: U \to W$, with the rest of the data implicit.
Over any closed point of $W$, we obtain a $k$-cycle on $X$ by taking the cycle underlying the corresponding fiber of $p$; we call these cycles the members of the family.

\begin{defn} \label{mcdefn}
Let $X$ be a projective variety and let $W$ be a variety.  Suppose that $U\subset W \times X$ is a subscheme and let $p: U \to W$ and $s: U \to X$ denote the projection maps.  The mobility count $\mc(p)$ of the morphism $p$ is the maximum non-negative integer $b$ such that the map
\begin{equation*}
U \times_{W} U \times_{W} \ldots \times_{W} U \xrightarrow{s \times s \times \ldots \times s} X \times X \times \ldots \times X
\end{equation*}
is dominant, where we have $b$ terms in the product on each side.  (If the map is dominant for every positive integer $b$, we set $\mc(p) = \infty$.)

For $\alpha \in N_{k}(X)_{\mathbb{Z}}$, the mobility count of $\alpha$, denoted $\mc(\alpha)$, is defined to be the largest mobility count of any family of effective cycles representing $\alpha$.
\end{defn}

\subsection{Geometry of families}

It will often be helpful to replace a family $p: U \to W$ by a slightly modified version.  We list briefly several possible changes.  We do not describe the constructions formally; most are explained more carefully in \cite{lehmann16}.
\begin{itemize}
\item Families of cycles admit proper pushforwards and flat pullbacks.  To perform such an operation, one does the corresponding operation on the cycle $\sum a_{i} U_{i}$ underlying the family; after passing to a smaller open subset $W^{0} \subset W$ to ensure flatness, we obtain a new family of cycles.
\item Suppose given two families $p: U \to W$ and $q: S \to T$ of effective $k$-cycles.  We can define the family sum over an open subset of $W \times T$ which parametrizes a sum of a member of $p$ and a member of $q$.  By \cite[Lemma 4.9]{lehmann16}, the mobility count adds under this operation.
\item Suppose $p:U \to W$ is a family of effective $k$-cycles and $D$ is a divisor.  If the general member of $p$ has no component contained in $D$, then we can take generic intersections to define a family $p \cdot D$ of effective $(k-1)$-cycles.  We can also take generic intersections with a linear series of Cartier divisors $\mathcal{D}$.
\end{itemize}
There are also a couple constructions which ``improve'' the geometry of a family without changing it in a fundamental way.
\begin{itemize}
\item Let $p: U \to W$ be a family of effective cycles on $X$.  Using the closedness of the Chow scheme, one can show that there is a normal projective variety $W'$ that is birational to $W$ and a family of cycles $p': U' \to W'$ such that $p$ and $p'$ agree over an open subset of the base.  By \cite[Proposition 4.5]{lehmann16}, this operation does not change the mobility count of $p$.
\item Let $p: U \to W$ be a family of effective cycles on $X$ and suppose that $U$ is irreducible.  By base-changing the family via a suitable morphism $g: T \to W^{0}$ for $W^{0} \subset W$ open, we may ensure that the general fiber of the base-change family is irreducible (as geometric integrality of fibers is constructible on the base).  

It is a priori unclear whether this change can affect the mobility count.  However, by using a suitable family sum to ``glue'' the components back together one can construct a family whose cycle-theoretic fibers are the same as those for $p$ but whose irreducible components have generically irreducible fibers.  The mobility count of this modified family is at least as large as $p$.  With more care, one can perform an analogous change when $U$ consists of several components.
\end{itemize}

In sum, when working with families of maximal mobility count, there is no loss in assuming that our family lies over a projective normal base and that the general fibers of the restriction of $p$ to each component of $U$ are irreducible.

Suppose given a dominant generically finite map $\phi: X \dashrightarrow Y$.  Let $p: U \to W$ denote a family of effective $k$-cycles on $X$.  We define the strict transform family $f_{*}p$ by first removing all components of $U$ whose map to $X$ is not dominant, taking the strict transform of the remaining components to a resolution of $f$, and then pushing forward the members of the family (see \cite{lehmann16}).

\begin{thrm}[\cite{lehmann16} Lemma 4.8] \label{mobcountstricttransform}
Let $X$ be a projective variety and let $p: U \to W$ be a  family of effective $k$-cycles on $X$.  If $f: X \dashrightarrow Y$ is a dominant generically finite map, then $\mc(p) \leq \mc(f_{*}p)$.  If $f$ is furthermore birational, then $\mc(p) = \mc(f_{*}p)$.
\end{thrm}

More generally, given any dominant map $f: X \dashrightarrow Z$ which does not contract any effective cycles $V$ through a general point satisfying $m\alpha \succeq [V]$, we obtain a pushforward family which has at least as large a mobility count as the original family.

\subsection{Variant of the mobility count}

We record for later use a variant of the mobility count.  By a family of (closed) subschemes of $X$, we mean a closed subscheme $R \subset S \times X$, where $S$ is a variety and the projection $q: R \to S$ is surjective.

\begin{defn} \label{familymcdef}
Let $X$ be a projective variety, and let $q: R \to S$ be a fixed family of subschemes of $X$ equipped with a flat morphism $t: R \to X$.  Suppose that $p: U \to W$ is a family of effective $k$-cycles on $X$, and let $p': U' \to W^{0}$ denote the flat pullback family to $R$.  We define the mobility count $\mc(p;q)$ of $p$ with respect to the family $q$ to be the largest non-negative integer $b$ such that the map
\begin{equation*}
U' \times_{W^{0}} U' \times_{W^{0}} \ldots \times_{W^{0}} U' \to S \times S \times \ldots \times S
\end{equation*}
is dominant, where we have $b$ terms in the product on each side.  (If the map is dominant for every positive integer $b$, we set $\mc(p;q) = \infty$.)
\end{defn}

Conceptually, $\mc(p;q)$ represents how many general elements of $q$ can be intersected by members of the family $p$.  One could of course define an analogous notion where $t$ is not assumed to be flat, but this situation is the only one we will need.

\begin{lem} \label{familymcbound}
Let $X$ be a projective variety of dimension $n$ with a fixed very ample divisor $A$.  Let  $q: R \to S$ be a family of equidimensional codimension $r$ subschemes of $X$ equipped with a flat morphism $t: R \to X$.  Consider a family $p: U \to W$ of effective $k$-cycles where $r > k$.  Then
\begin{equation*}
\mc(p;q) \leq 2^{kr+3r} ([p] \cdot A^{k} + 1)^{r/r-k}.
\end{equation*}
\end{lem}

The goal of this lemma is the exponent of $r/r-k$ in the degree of $[p]$; there has been no attempt to optimize the leading constant.

\begin{proof}
We may assume that $\mc(p;q) > 0$.  Just as in Definition \ref{familymcdef}, let $p': U' \to W^{0}$ be the flat pullback family of $p$.  Then we have a dominant map
\begin{equation*}
U' \times_{W^{0}} U' \times_{W^{0}} \ldots \times_{W^{0}} U' \to S \times S \times \ldots \times S
\end{equation*}
where there are $\mc(p;q)$ terms on both sides.  Consider the map
\begin{equation*}
U' \times_{W^{0}} U' \times_{W^{0}} \ldots \times_{W^{0}} U' \to R \times R \times \ldots \times R
\end{equation*}
and let $V_{R}$ denote the image.
By composing with the flat map $R \to X$ we obtain
\begin{equation*}
U' \times_{W^{0}} U' \times_{W^{0}} \ldots \times_{W^{0}} U' \to X \times X \times \ldots \times X.
\end{equation*}
Let $f: X \dashrightarrow \mathbb{P}^{r}$ denote the rational map defined by a general $(r+1)$-dimensional subspace of $H^{0}(X,A)$.  
We claim that the induced rational map
\begin{equation*}
U' \times_{W^{0}} U' \times_{W^{0}} \ldots \times_{W^{0}} U' \dashrightarrow \mathbb{P}^{r} \times \mathbb{P}^{r} \times \ldots \times \mathbb{P}^{r}
\end{equation*}
is dominant, where there are $\mc(p;q)$ factors on each side.  We argue inductively on the number of factors.  Suppose we fix a general fiber of the map $R^{\times \mc(p;q)} \to S^{\times \mc(p;q)-1}$.  The intersection of $V_{R}$ with this fiber dominates $S$ under the first projection, and a dimension count (using flatness of $R \to X$) shows that this set must meet the pullback of a general complete intersection variety $Q = A^{n-r}$ from the first factor $X$.  Varying the fiber we see that $V_{R} \cap f_{1}^{-1}(Q)$ maps dominantly onto $S^{\mc(p;q)-1}$.  Repeating this argument inductively, we see that $V_{R}$ must intersect the intersection of the pullbacks of a general $Q$ from all the factors, which is equivalent to the dominance of the map.

By construction, this map factors through 
\begin{equation*}
U|_{W^{0}} \times_{W^{0}} U|_{W^{0}} \times_{W^{0}} \ldots \times_{W^{0}} U|_{W^{0}} \dashrightarrow \mathbb{P}^{r} \times \mathbb{P}^{r} \times \ldots \times \mathbb{P}^{r}
\end{equation*}
which is then itself dominant.  In fact, even if we replace $W^{0}$ by a smaller open subset, this map will still be dominant; this follows from the argument of \cite[Proposition 4.5]{lehmann16}.

By generality of $f$, we can pushforward $p$ to define a family of $k$-cycles on $\mathbb{P}^{r}$.  Note that this image family has degree $[p] \cdot A^{k}$.  Furthermore, since we obtain a dominant map above even when shrinking $W^{0}$ to the locus where the pushforward family is defined, the pushforward family has mobility count at least $\mc(p;q)$.  One then applies \cite[Theorem 5.12]{lehmann16} to bound mobility counts on projective space.
\end{proof}

\section{Basic properties}

We next verify some basic properties of the Iitaka dimension.  The following Lemma \ref{iitakadimhom} shows that, just as for divisors, the Iitaka dimension is invariant under rescaling.  In particular, we can naturally extend the Iitaka dimension to any class $\alpha \in N_{k}(X)_{\mathbb{Q}}$.

\begin{lem} \label{iitakadimhom}
Let $X$ be a projective variety and let $\alpha \in N_{k}(X)_{\mathbb{Z}}$.  Then for any positive integer $c$ we have $\kappa(\alpha) = \kappa(c\alpha)$.
\end{lem}

\begin{proof}
Fix a positive real number $r$, and define the function $g_{r}: N_{k}(X)_{\mathbb{Z}} \to \mathbb{R} \cup \{ \infty \}$ by
\begin{equation*}
g_{r}(\alpha) = \limsup_{m \to \infty} \frac{\mc(m\alpha)}{m^{r}}.
\end{equation*}
It suffices to show that $c^{r}g_{r}(\alpha) = g_{r}(c\alpha)$.  This is a consequence of the following Lemma \ref{lazlemma} applied to the function $f: \mathbb{N} \to \mathbb{R}_{\geq 0}$ which sends $m \mapsto \mc(m\alpha)$.  (Note that the hypothesis of Lemma \ref{lazlemma} is verified by using the additivity of the mobility count of the family sum.)
\end{proof}

\begin{lem}[\cite{lazarsfeld04} Lemma 2.2.38] \label{lazlemma}
Let $f: \mathbb{N} \to \mathbb{R}_{\geq 0}$ be a function.  Suppose that for any $r,s \in \mathbb{N}$ with $f(r) > 0$ we have that $f(r+s) \geq f(s)$.
Then for any $k \in \mathbb{R}_{>0}$ the function $g: \mathbb{N} \to \mathbb{R} \cup \{ \infty \}$ defined by
\begin{equation*}
g(r) := \limsup_{m \to \infty} \frac{f(mr)}{m^{k}}
\end{equation*}
satisfies $g(cr) = c^{k}g(r)$ for any $c,r \in \mathbb{N}$.
\end{lem}

\begin{rmk}
Although \cite[Lemma 2.2.38]{lazarsfeld04} only explicitly address the volume function, the essential content of the proof is the more general statement above.
\end{rmk}

\begin{prop}
Let $X$ be a projective variety of dimension $n$ and let $\alpha \in N_{k}(X)_{\mathbb{Z}}$.  Then
\begin{equation*}
\kappa(\alpha) \in \{ -\infty \} \cup \{ 0 \} \cup [n-k,n].
\end{equation*}
\end{prop}

\begin{proof}
The upper bound $\kappa(\alpha) \leq n$ is proved by \cite[Proposition 5.1]{lehmann16}.

If every positive multiple of $\alpha$ has vanishing mobility count, then $\kappa(\alpha) \in \{ -\infty, 0\}$.  Otherwise for some positive integer $s$ we have $\mc(s\alpha) > 0$.  Using additivity of the mobility count under family sums we see that $\mc(ms\alpha) \geq m \mc(s\alpha)$ so that $\kappa(\alpha) \geq n-k$.
\end{proof}

\subsection{Concentration of mobility count}

The following somewhat technical result shows that an increasing mobility count must be concentrated on families of irreducible cycles.

\begin{lem} \label{irrfamconcentrates}
Let $X$ be a projective variety of dimension $n$ and let $\alpha \in N_{k}(X)_{\mathbb{Q}}$. Suppose that $\kappa(\alpha) > n-k$.  Fix an ample divisor $H$ on $X$.  Then for any positive integer $M$, any positive constant $C$, and any sufficiently small positive $\epsilon$, there is some integer $m > M$ and an irreducible family of $k$-cycles $p$ such that $\mc(p) > C(H^{k} \cdot p)^{\frac{\kappa(\alpha)}{n-k}-\epsilon}$ and $m\alpha - [p]$ is the class of an effective $\mathbb{Z}$-cycle.
\end{lem}

\begin{proof}
Suppose for a contradiction that we can choose an $M$, $C$, and $\epsilon$ violating the condition.  For any $m>M$ such that $m\alpha \in N_{k}(X)_{\mathbb{Z}}$, choose finitely many irreducible families $p_{i}$ whose family sum gives a family of cycles representing $\alpha$ of maximal mobility count.  Since each class $m\alpha-[p_{i}]$ represents an effective $\mathbb{Z}$-cycle, we find that for every sufficiently large $m$
\begin{align*}
\mc(m\alpha) & = \sum_{i} \mc(p_{i}) \\
& \leq \sum_{i} C(H^{k} \cdot p_{i})^{\frac{\kappa(\alpha)}{n-k} - \epsilon} \\
& \leq C(H^{k} \cdot \alpha)^{\frac{\kappa(\alpha)}{n-k} - \epsilon} m^{\frac{\kappa(\alpha)}{n-k} - \epsilon}
\end{align*}
where the last inequality follows from convexity.  This is a contradiction to the expected growth rate of $\mc(m\alpha)$.
\end{proof}

\subsection{Ample intersections}

We next analyze the behavior of $\kappa(X)$ under intersections with ample divisors.

\begin{lem} \label{linearseriesintest}
Let $X$ be a projective variety of dimension $n$.  Suppose that $p$ is a family of irreducible $k$-cycles and $r: \mathcal{D} \to V$ is a linear series of effective Cartier divisors.  Then
\begin{equation*}
\mc(p \cdot \mathcal{D}) \geq \min\{ \mc(p) - 1, \mc(r) \}.
\end{equation*}
\end{lem}

Recall that $p \cdot \mathcal{D}$ denotes the family of cycles which are intersections of general elements of $p$ with general elements of $\mathcal{D}$.

\begin{proof}
It suffices to consider the case when $\mc(p) > 1$.  Suppose we fix a Cartier divisor $D$ through $\mc(r)$ general points $x_{i}$ of $X$.  There is an irreducible element $Z$ of the family $p$ containing a general point of $X - \Supp(D)$ and any $\min\{ \mc(p)-1,\mc(r) \}$ of the $x_{i}$.  Since $Z$ is not contained in $\Supp(D)$, we can take cycle-theoretic intersections to obtain the desired family.
\end{proof}

\begin{prop} \label{ampleintandiitakadim}
Let $X$ be a projective variety of dimension $n$ and let $\alpha \in N_{k}(X)_{\mathbb{Q}}$.  Suppose that $H$ is an ample $\mathbb{Q}$-Cartier divisor on $X$.  We have $\kappa(\alpha \cdot H) \geq \kappa(\alpha)$, with strict inequality unless $\kappa(\alpha) \in \{ -\infty, 0, n \}$.
\end{prop}

\begin{proof}
Since the Iitaka dimension is invariant under rescaling, it suffices to prove this when $H$ is an ample Cartier divisor satisfying the condition
\begin{equation*}
\mc(|mH|) \geq cm^{n}
\end{equation*}
for some constant $c>1$.

The statement is obvious when $\kappa(\alpha) = -\infty$.  We next analyze the special cases $\kappa(\alpha) \in \{0, n \}$.  If some multiple of $\alpha$ is represented by an effective cycle, then by Lemma \ref{intprop} the same is true for some multiple of $\alpha \cdot H$, showing the inequality when $\kappa(\alpha)=0$.
If $\alpha$ is big, then $\alpha \cdot H$ is also big by Lemma \ref{intprop}, showing the inequality when $\kappa(\alpha) =n$. 

Next suppose that $\kappa(\alpha) = n-k$.  Clearly we can find effective cycles with class proportional to $\alpha \cdot H$ through any general point of $X$.  Thus $\kappa(\alpha \cdot H) \geq n-k+1 > \kappa(\alpha)$.

Finally, consider the case when $n-k < \kappa(\alpha) < n$.  Fix a positive constant $C$ and an $\epsilon > 0$.  Suppose that $p_{i}: U_{i} \to W$ are the irreducible families of $k$-cycles composing a family $p$ representing $m\alpha$ of maximal mobility count; as described in Section \ref{familysec}, we may assume that each $U_{i}$ has irreducible generic fiber.  For $m$ sufficiently large, we have
\begin{equation*}
\mc(|\lceil m^{\frac{\kappa(\alpha)}{n(n-k)}} \rceil H |) \geq cm^{\frac{\kappa(\alpha)}{n-k}} > Cm^{\frac{\kappa(\alpha)-\epsilon}{n-k}}.
\end{equation*}
Note that $p$ can have at most $H^{k} \cdot m\alpha$ components.  Thus for $m$ sufficiently large, we have
\begin{align*}
\mc(m \lceil m^{\frac{\kappa(\alpha)}{n(n-k)}} \rceil \alpha \cdot H) & \geq \sum_{i} \mc(p_{i} \cdot |\lceil m^{\frac{\kappa(\alpha)}{n(n-k)}} \rceil H |) \\
& \geq \sum_{i} \min\{ \mc(p_{i}) -1, \mc(|\lceil m^{\frac{\kappa(\alpha)}{n(n-k)}}\rceil H |) \} \textrm{ by Lemma \ref{linearseriesintest} } \\
& \geq Cm^{\frac{\kappa(\alpha) - \epsilon}{n-k}} - m(H^{k} \cdot \alpha).
\end{align*}
Note that $m \cdot \lceil m^{\kappa(\alpha)/n(n-k)} \rceil \leq 2 m^{1+\kappa(\alpha)/n(n-k)}$.  Thus, by renormalizing and taking roots, for any positive constant $\widetilde{C}$ and any $\epsilon > 0$ we have for $m$ sufficiently large
\begin{equation*}
\mc(m\alpha \cdot H) \geq \widetilde{C}m^{\frac{\kappa(\alpha)-\epsilon}{n-k+\frac{\kappa(\alpha)}{n}}}.
\end{equation*}
The codimension of $\alpha \cdot H$ is $n-k+1$.  Since the equality above is true for any positive $\widetilde{C}$ and any sufficiently small $\epsilon$, by taking limits we find
\begin{equation*}
\kappa(\alpha \cdot H) \geq \frac{n-k+1}{n-k+\frac{\kappa(\alpha)}{n}} \kappa(\alpha).
\end{equation*}
\end{proof}

\subsection{Birational behavior of the Iitaka dimension}

Suppose that $\phi: Y \to X$ is a birational morphism of projective varieties.  Given any class $\beta \in N_{k}(Y)_{\mathbb{Q}}$, Theorem \ref{mobcountstricttransform} shows that $\kappa(\phi_{*}\beta) \geq \kappa(\beta)$.  In this section, we address the opposite question: given a class $\alpha \in N_{k}(X)_{\mathbb{Q}}$, what are the possible Iitaka dimensions of classes $\beta$ satisfying $\phi_{*}\beta = \alpha$?

\begin{thrm} \label{birbeh}
Let $\phi: Y \to X$ be a birational morphism of projective varieties.  Suppose $\alpha \in N_{k}(X)_{\mathbb{Q}}$.  Then there is a class $\beta \in N_{k}(Y)_{\mathbb{Q}}$ satisfying $\phi_{*}\beta = \alpha$ and $\kappa(\beta) = \kappa(\alpha)$.
\end{thrm}

Before proving this theorem, we need to recall some results of \cite{fl13} concerning the movable cone and its behavior under birational maps.

\begin{defn}
The movable cone $\Mov_{k}(X)$ is the closure of the cone generated by classes of irreducible subvarieties which deform to cover $X$.  We say that $\alpha \in N_{k}(X)$ is movable if it lies in $\Mov_{k}(X)$.
\end{defn}

\begin{lem}[\cite{fl13} Corollary 6.6] \label{compactnessofmov}
Let $\phi: Y \to X$ be a birational morphism of projective varieties.  Fix a class $\alpha \in \Eff_{k}(X)$.  Then the set of classes
\begin{equation*}
\mathcal{S} := \{ \beta \in \Mov_{k}(Y) | \phi_{*}\beta \preceq \alpha \}
\end{equation*}
is compact.
\end{lem}

We also need:

\begin{lem} \label{conelem}
Let $M$ be a real vector space, $L \subset M$ a full rank lattice and $T \subset L$ a subsemigroup which generates $L$.  For any compact subset $S \subset M$, there is an element $\widetilde{\beta} \in T$ such that
\begin{equation*}
m(\widetilde{\beta} - S) \cap L \subset T
\end{equation*}
for any positive integer $m$.
\end{lem}

\begin{proof}
Let $C$ denote the closure of the cone in $M$ generated by elements of $T$.  Since $C$ is full-dimensional, it is clear that there is a class $\gamma \in T$ such that $\gamma - S \subset C^{\circ}$.  One can then choose a subcone $C' \subset C$ which is finitely generated by a subset of $T$ which still generates $L$ and such that $\gamma - S \subset C'$.  Note that $m(\gamma - S) \subset C'$ for any positive integer $m$.  One can then conclude by the argument of \cite[Lemma 4.13]{fl13} the existence of a $\beta \in T$ such that
\begin{equation*}
\beta + (m(\gamma - S) \cap L) \subset T
\end{equation*}
for any positive integer $m$.  Set $\widetilde{\beta} = \beta + \gamma$.
\end{proof}

\begin{proof}[Proof of Theorem \ref{birbeh}:]
Without loss of generality we may suppose $\alpha \in \Eff_{k}(X)_{\mathbb{Z}}$.   Define:
\begin{itemize}
\item $T_{Y} \subset N_{k}(Y)_{\mathbb{Z}}$ to be the subsemigroup consisting of effective classes $\xi$ such that $m\alpha - \phi_{*}\xi$ is an effective class for some $m>0$.
\item $L_{Y}$ for the sublattice of $N_{k}(Y)_{\mathbb{Z}}$ generated by $T_{Y}$.
\item $M_{Y}$ the subspace of $N_{k}(Y)$ spanned by $T_{Y}$. 
\end{itemize}
Let $\mathcal{S}$ be the compact set as constructed in Lemma \ref{compactnessofmov}.  Then $\mathcal{S} \cap M_{Y}$ is also compact.  By Lemma \ref{conelem}, there is some class $\widetilde{\beta} \in T_{Y}$ such that $m(\widetilde{\beta} - (\mathcal{S} \cap M_{Y})) \cap L_{Y} \subset T_{Y}$ for any positive integer $m$.

Fix a positive integer $m$ and a family of effective $k$-cycles $p: U \to W$ of class $m\alpha$ of maximal mobility count. 
Remove all non-dominant components of $U$ and consider the strict transform family $q$ on $Y$.  Note that $[q] \in m\mathcal{S} \cap T_{Y}$.  Thus $m \widetilde{\beta} - [q] \in m(\widetilde{\beta} - (\mathcal{S} \cap M_{Y})) \cap L_{Y}$ is an effective class.  Consequently
\begin{equation*}
\mc(m \widetilde{\beta}) \geq \mc(q) = \mc(p) = \mc(m \alpha).
\end{equation*}
Since $\widetilde{\beta} \in T_{Y}$, by definition there is some positive integer $c$ such that $c\alpha - \phi_{*}\widetilde{\beta}$ is an effective class $\nu$.  By \cite[Proposition 3.21]{fl14} there is an effective $\mathbb{Q}$-class $\mu$ such that $\phi_{*}\mu= \nu$; let $b$ be a positive integer such that $b\mu$ is an effective class.  Set $\beta := \frac{1}{c}(\mu + \widetilde{\beta})$.  Then $\mc(cbm\beta) \geq \mc(m\alpha)$ for all positive integers $m$, and we obtain $\kappa(\beta) \geq \kappa(\alpha)$ by the invariance of Iitaka dimensions under rescaling.  The reverse inequality follows from Theorem \ref{mobcountstricttransform}. 
\end{proof}

\subsection{Extremal rays}

There are some techniques which one can sometimes apply to give upper bounds on the Iitaka dimension of a class on an extremal ray.

\begin{prop}
Let $X$ be a smooth projective variety of dimension $n \geq 3$.  Suppose that $\alpha \in \Eff_{n-2}(X)_{\mathbb{Q}}$ spans an extremal ray and that there is an ample divisor $A$ such that $\alpha \not \in A \cdot N^{1}(X)$.  Then $\kappa(\alpha) \leq n-1$.
\end{prop}

\begin{proof}
Fix a positive integer $q$ so that $qA$ is very ample, and let $H$ be a very general member of $|qA|$.  By the Lefschetz hyperplane theorem we see that if $Z \subset H$ is an effective divisor then its class in $X$ lies in $A \cdot N^{1}(X)$.  In particular, since $\alpha$ spans an extremal ray, the class $m\alpha - [Z]$ is not pseudo-effective for any positive integer $m$ and for any such choice of $Z$.

Let $p: U \to W$ be a family of effective cycles representing $m\alpha$ of maximal Iitaka dimension.  By the above argument, there is no component of any member of the family which is contained in $H$.  Thus, by taking the intersection of these cycles with the divisor $H$ we obtain a family $p_{H}$ of $(n-3)$-cycles on $H$ of class $m \alpha \cdot H$.  By arguing as in \cite[Theorem 5.12]{lehmann16}, we see that $\mc_{X}(p) \leq \mc_{H}(p_{H})$.  As $H$ has dimension $n-1$ and the class of $p_{H}$ grows linearly in $m$, the mobility count of $p_{H}$ is bounded above by $Cm^{n-1/n-k}$ for some constant $C$.
\end{proof}

The argument above clearly extends to other codimensions when $X$ satisfies a suitable Lefschetz theorem.

\section{Iitaka dimension of divisors} \label{divisorsec}

We next show that for divisors the Iitaka dimension is an integer.  This is a familiar fact for the classical Iitaka dimension defined by sections; we verify that the numerical version has similar behavior.  To differentiate the two, we let $\kapclas(D)$ denote the classical Iitaka dimension of a Cartier divisor $D$.

To study the mobility count of divisors, it is often useful to reformulate the definition as follows.  Suppose that $X$ is smooth and that $p: U \to W$ is a family of effective divisors on $X$ with $W$ normal.  We obtain an induced map $ch: W \to \Chow(X)$ and the mobility count of $p$ coincides with the dimension of $\overline{ch(W)}$.  We will frequently use this interpretation in this section.  

\begin{lem} \label{cartiersectionestimate}
Let $X$ be a smooth projective variety of dimension $n$ and let $D$ be a Cartier divisor on $X$ with $\kapclas(D) = r$.  Let $A$ be the pullback of a very ample divisor under a birational map and let $s$ be a positive integer such that $D \cdot A^{n-1} < sA^{n}$.  Then
\begin{equation*}
h^{0}(X,D) < s^{r}A^{n} + 1.
\end{equation*}
\end{lem}

\begin{proof}
Let $\phi: X' \to X$ be a birational map resolving the linear series $|D|$.  Let $M$ denote a divisor in the basepoint free part of $\phi^{*}|D|$ and let $\pi: X' \to Z$ be the morphism induced by $|M|$.  Note that $\dim(Z) \leq r$; we may also assume that $\dim(Z) \geq 1$ since otherwise the desired inequality is clear.  Since
\begin{equation*}
M \cdot \phi^{*}A^{n-1} \leq D \cdot A^{n-1}
\end{equation*}
it suffices to prove the statement for $M$.  Fix $n-r$ general elements of the linear series $A_{1},\ldots,A_{n-r} \in |A|$ and let $W_{i}$ denote the scheme-theoretic intersection of the first $i$ of these.  Note that $M|_{W_{i}}$ is not big for $i<n-r$; using the LES for restriction of sections inductively one sees that $h^{0}(X',M) \leq h^{0}(W_{n-r},M)$.  But by another easy inductive argument using a LES of sections and cutting down by hyperplanes the latter is bounded above by $s^{r}A^{n}+1$.
\end{proof}

\begin{thrm}
Let $X$ be a smooth projective variety of dimension $n$ and let $\alpha \in N_{n-1}(X)_{\mathbb{Q}}$.  Then
\begin{equation*}
\kappa(\alpha) = \sup_{L \in \mathrm{Div}(X) \otimes \mathbb{Q}, [L] = \alpha} \kapclas(L).
\end{equation*}
In particular, $\kappa(\alpha) \in \{-\infty \} \cup \mathbb{Z}_{\geq 0}$.
\end{thrm}

\begin{proof}
By homogeneity it suffices to consider the case when $\alpha \in N_{n-1}(X)_{\mathbb{Z}}$.  Set $r$ to be the maximum over all Iitaka dimensions as in the right hand side of the statement above.  The inequality $\kappa(\alpha) \geq r$ is clear. 

Conversely, let $P(X)$ denote the dual of the Albanese variety of $X$.  Fix a very ample divisor $A$ on $X$ and choose a positive integer $s$ such that
\begin{equation*}
\alpha \cdot A^{n-1} \leq s A^{n}.
\end{equation*}
Let $p: U \to W$ denote a family of effective divisors on $X$ representing $m\alpha$.  Then $p$ induces a rational map $W \dashrightarrow P(X)$ defined on the normal locus $W^{\circ} \subset W$.  We see that $\dim(\overline{ch_{p}(W^{\circ})}) \leq \dim(P(X)) + \chdim(p|_{F})$ where $F$ is a general fiber of the map from $W$ to $P(X)$.  In particular for any component $F'$ of $F$ we have that $p|_{F'}$ is a family of rationally equivalent effective divisors.  Thus by Lemma \ref{cartiersectionestimate} we have
\begin{equation*}
\mc(p) \leq \dim(P(X)) + m^{r}s^{r}A^{n} + 1
\end{equation*}
and the reverse inequality follows.
\end{proof}

By Theorem \ref{birbeh} we deduce:

\begin{thrm}
Let $X$ be a projective variety of dimension $n$ and let $\alpha \in N_{n-1}(X)_{\mathbb{Q}}$.  Let $\phi: X' \to X$ be a smooth birational model.  Then
\begin{equation*}
\kappa(\alpha) = \sup_{L \in \mathrm{Div}(X') \otimes \mathbb{Q}, \phi_{*}[L] = \alpha} \kapclas(L).
\end{equation*}
In particular, $\kappa(\alpha) \in \{-\infty \} \cup \mathbb{Z}_{\geq 0}$.
\end{thrm}

\begin{rmk}
When $X$ is normal, one can define the Iitaka dimension of any Weil divisor $D$ capturing the asymptotic growth rate of sections of the rank $1$ reflexive sheaves $\mathcal{O}_{X}(mD)$ analogously with the Cartier divisor case.  There is a Cartier divisor $D'$ on a birational model of $X$ with exactly the same behavior of sections (see for example \cite[Lemma 3.3]{fkl15}), so that the Iitaka dimension is still integer-valued.  For normal varieties the Iitaka dimension of a Weil divisor numerical class coincides with the maximal Iitaka dimension of any Weil divisor representing the class by essentially the same argument.
\end{rmk}

\section{Contracted classes}

Suppose that $\pi: X \to Z$ is a morphism and $\alpha \in \Eff_{k}(X)$ satisfies $\pi_{*}\alpha = 0$.  The goal of this section is to bound the Iitaka dimension of $\alpha$ in terms of the geometry of the map $\pi$.

\begin{defn}
Let $\pi: X \to Z$ be a surjective morphism of projective varieties and let $\alpha \in \Eff_{k}(X)$.  Fix an ample divisor $A$ on $Z$.  The $\pi$-contractibility index of $\alpha$ is defined to be the largest non-negative integer $c \leq k+1$ such that $\alpha \cdot \pi^{*}A^{k-c+1}=0$.  This definition is independent of the choice of $A$.
\end{defn}

The basic properties of the contractibility index are described by \cite[Section 4.2]{fl15}:
\begin{itemize}
\item The contractibility index of $\alpha$ is positive precisely when $\pi_{*}\alpha=0$.
\item The contractibility index is at most $k+1$, with equality only for the $0$ class.
\item The contractibility index is at least $k-\dim Z$.  If the contractibility index is larger than this minimum value, then no effective cycle of class $\alpha$ surjects onto $Z$.
\end{itemize}

\begin{exmple}
If $V$ is an irreducible subvariety of $X$, then the contractibility index of $[V]$ is the same as $\mathrm{reldim}(\pi|_{V})$.
\end{exmple}

The Iitaka dimension of $\alpha$ exhibits two different behaviors, based on whether the contractibility index $c$ is smaller or larger than the relative dimension of $\pi$.  First suppose that $c > \mathrm{reldim}(\pi)$.  \cite{fl15} calls such classes ``$\pi$-exceptional'' and shows that they are rigid in a strong sense.  In particular:

\begin{lem}
Let $\pi: X \to Z$ be a surjective morphism of projective varieties of relative dimension $e$.  Suppose that $\alpha \in \Eff_{k}(X)_{\mathbb{Z}}$ has contractibility index $c$ and that $c > e$.  Then $\kappa(\alpha) \leq 0$.
\end{lem}

\begin{proof}
Arguing as in \cite[Lemma 4.11]{fl15}, there is a proper closed subset $E$ in $X$ such that any effective cycle represented by a multiple of $\alpha$ is contained in $E$.  Thus $\mc(m\alpha) = 0$ for all positive integers $m$.
\end{proof}

The case when $c = \mathrm{reldim}(\pi)$ is handled by the following naive bound.

\begin{lem} \label{naivebound}
Let $\pi: X \to Z$ be a surjective morphism from a projective variety of dimension $n$ to a projective variety of dimension $d$.  Suppose that $\alpha \in \Eff_{k}(X)$ has contractibility index $c$ and that $k - \dim(Z) < c \leq n-d$.  Then $\kappa(\alpha) \leq (n-k) \cdot \frac{d}{d-k+c}$.
\end{lem}

Note that when $c = n-d$, this simplifies to $\kappa(\alpha) \leq d$ as desired.

\begin{proof}
The statement is clear if $\kappa(\alpha) \leq n-k$, so we may assume otherwise.  Fix an ample divisor $A$ on $X$ and an ample divisor $H$ on $Z$.

Let $m$ be a positive integer and let $p_{m}$ be a family of effective cycles representing $m\alpha$ with maximal mobility count.  The image of a cycle in the family is a subscheme of $Z$; every component has dimension at most $k-c < \dim(Z)$ and has $H$-degree bounded linearly in terms of $\alpha \cdot A^{k}$.  It is clear that $\mc(p_{m})$ is bounded above by the mobility count of the images.  Using  \cite{lehmann16} to bound the mobility count of the images we see that there is some constant $C$ such that $\mc(p_{m}) \leq C m^{d/d-k+c}$.
\end{proof}

We next turn our attention to the case when $c$ is close to $k$.

\begin{lem} \label{reldimequal}
Let $\pi: X \to Z$ be a surjective morphism of projective varieties of relative dimension $e$.  Suppose that $\alpha \in \Eff_{k}(X)_{\mathbb{Z}}$ has contractibility index $c = k$.  Then $\kappa(\alpha) \leq n-c$.
\end{lem}

\begin{proof}
Let $V$ be an effective cycle representing $m\alpha$ through $b$ general points of $X$.  Then $\pi(V)$ is a union of points on $Z$, which contains $b$ general points as a subset.  Since the cardinality of $\pi(V)$ can only grow linearly with $m$, we obtain the result.
\end{proof}

\begin{thrm} \label{contractedthrm}
Let $\pi: X \to Z$ be a morphism from a projective variety of dimension $n$ to a projective variety of dimension $d$.  Let $\alpha \in \Eff_{k}(X)$ be a class of $\pi$-contractibility index $c$.  If $c = k-1$, then $\kappa(\alpha) \leq n-c$.
\end{thrm}

\begin{proof}
Set $e = n-d$.

We start with several reductions.  Let $\phi: X' \to X$ be a birational model and let $\beta \in \Eff_{k}(X)$ be a class such that $\pi_{*}\beta = \alpha$.  Then the contractibility index of $\beta$ for $\pi \circ \phi$ is still $c$.  Thus by Theorem \ref{birbeh} it suffices to replace $X$ by any higher birational model.

In particular, we may suppose that $X$ admits a map $\rho: X \to Y$ where $Y$ is a variety of dimension $e$ and $\rho|_{F}$ is generically finite for a general fiber $F$ of $\pi$.  Such $X$ naturally carries a generically finite map surjecting onto $Z \times Y$, and hence also to $\mathbb{P}^{d} \times \mathbb{P}^{e}$.  Since the Iitaka dimension can only increase under pushforward to this variety, and the contractibility index can also only increase, it suffices to consider the case when $X = \mathbb{P}^{d} \times \mathbb{P}^{e}$ and $\pi$ is the first projection map.  In this setting, we let $\rho$ denote the second projection map, $A = \pi^{*}\mathcal{O}(1)$, $H = \rho^{*}\mathcal{O}(1)$.

Suppose that $p: U \to W$ is a family of irreducible cycles on $X$ such that $[p] \preceq m\alpha$.  Each irreducible cycle $V$ in the family is mapped to a subvariety $V' \subset \mathbb{P}^{d}$ of dimension $k-c=1$.  We associate the following invariants to $V$:
\begin{itemize}
\item $\sigma$ is the degree of $V'$
\item $\tau$ is the degree of $V \cap F \subset \mathbb{P}^{e}$, where $F = \pi^{-1}(p)$ for a general point $p \in V'$.
\end{itemize}

Since $V'$ is a curve we have that the base-change of $V$ to the normalization $V'^{n}$ of $V'$ is flat, so we can equally well think of each member of our family $V$ as a family of cycles in projective space defined by a map $f: V'^{n} \to \Chow_{\tau,c}(\mathbb{P}^{e})$, where $\Chow_{\tau,c}$ denotes the subvarieties of dimension $c$ and degree $\tau$.  We can realize $\Chow_{\tau,c}(\mathbb{P}^{e})$ as a subvariety of $\mathbb{P}H^{0}(\mathbb{G},L^{\otimes \tau})$ where $\mathbb{G} = G(e-c,e+1)$ is the Grassmannian and $L$ is the pullback of $\mathcal{O}(1)$ under the Pl\"ucker embedding.  Let $M$ denote the very ample divisor on $\Chow$ induced by pulling back $\mathcal{O}(1)$ from this projective space.  Let $\nu_{1}$ and $\nu_{2}$ denote the projection maps on $\mathbb{P}^{d} \times \Chow_{\tau,c}(\mathbb{P}^{e})$ and let $T \subset \mathbb{P}^{d} \times \Chow_{\tau,c}(\mathbb{P}^{e})$ denote the image of $V'^{n}$.  Note that
\begin{equation*}
T \cdot \nu_{1}^{*}\mathcal{O}(1) = \sigma \leq V \cdot A \cdot H^{k-1} \qquad \textrm{and} \qquad T \cdot \nu_{2}^{*}M = V \cdot H^{k}
\end{equation*}
so that degree of $T$ against the ample divisor $\nu_{1}^{*}\mathcal{O}(1) + \nu_{2}^{*}M$ is bounded linearly in terms of the class of $\alpha$.

Fix a component $\mathcal{C} \subset \Chow_{\tau,c}(\mathbb{P}^{e})$ which contains the image of $T$.  Let $q: R \to \mathbb{P}^{e}$ denote the family of subschemes of $\mathcal{C}$ where the fiber over $x \in \mathbb{P}^{e}$ parametrizes all cycles containing $x$.  The induced map $R \to \mathcal{C}$ is flat by equivariance.  We let $\widetilde{q}: \mathbb{P}^{d} \times R \to \mathbb{P}^{d} \times \mathbb{P}^{e}$ denote the corresponding family on the product $\mathbb{P}^{d} \times \mathcal{C}$.  The subvarieties parametrized by $\widetilde{q}$ have codimension $d+e-c = n-c$.  Note that
\begin{equation*}
\mc_{X}(p) = \mc_{\mathbb{P}^{d} \times \mathcal{C}}(\widetilde{p};\widetilde{q})
\end{equation*}
where $\widetilde{p}$ is the family of $(k-c)$-dimensional subvarieties $T$ on $\mathbb{P}^{d} \times \mathcal{C}$ induced by $p$.  Consider the very ample divisor $\widetilde{A} = \nu_{1}^{*}\mathcal{O}(1) + \nu_{2}^{*}M$.  By Lemma \ref{familymcbound}, we have
\begin{equation*}
\mc_{X}(p) \leq 2^{(k-c)(n-c)+3(n-c)} \left( [\widetilde{p}] \cdot \widetilde{A}^{k-c} + 1 \right)^{n-c/n-c-(k-c)}
\end{equation*}
By the intersection calculations above, $[\widetilde{p}] \cdot \widetilde{A}^{k-c}$ is bounded linearly in terms of $[p] \cdot (A+H)^{k}$.  Thus we see that for such an irreducible family $p$, there is some constant $C$ such that $\mc(p) \leq Cm^{n-c/n-k}$.  We conclude by Lemma \ref{irrfamconcentrates} that $\kappa(\alpha) \leq n-c$.
\end{proof}

The upper bound proposed by Conjecture \ref{contractedconj} is optimal in a strong sense: given any morphism $\pi: X \to Z$ and any choice of $k,c$ satisfying the necessary constraints, one can always find a class $\alpha \in \Eff_{k}(X)$ of contractibility index $c$ and Iitaka dimension $\geq n-c$.

\begin{exmple}
We first set parameters.  Choose integers $0 < d < n$ and integers $k,c$ such that $0 < k < n$ and $\min\{0,k-d\} < c < \min\{k+1,n-d\}$.

Suppose that $\pi: X \to Z$ is a morphism where $X$ has dimension $n$ and $Z$ has dimension $d$.  We construct a $k$-cycle class $\alpha$ on $X$ with contractibility index $c$ and with $\kappa(\alpha) \geq n-c$.  The first step is a reduction: suppose that there is a diagram
\begin{displaymath}
    \xymatrix{
        X' \ar[r] \ar[d]_{\pi'} & X \ar[d]^{\pi} \\
        Z' \ar[r]       & Z}
\end{displaymath}
where the horizontal maps are generically finite.  We claim that it is enough to construct a suitable class $\alpha'$ on $X'$ for $\pi'$, if we assume in addition that $\alpha'$ has a multiple represented by an irreducible cycle $V$ through a general point.  Indeed, the pushforward $\alpha$ on $X$ will have at least as large of an Iitaka dimension as $\alpha'$.  Furthermore, the contractibility index of $\alpha'$ is simply the relative dimension of $\pi'|_{V}$, and it is clear that this quantity is preserved by pushing forward $V$.

By applying Noether normalization to the function field $K(Z)$, after replacing $X$ and $Z$ by birational models we may assume that there is a diagram
\begin{displaymath}
    \xymatrix{
        X \ar[r]^{f} \ar[d]_{\pi} & W \ar[dl]_{g} \\
        Z       & }
\end{displaymath}
where $W$ has dimension $n-c$.  Let $H$ be an ample divisor on $W$ and set $\alpha = f^{*}H^{n-k}$.  Since $\alpha$ is represented by the preimages of a big $(k-c)$-cycle on $W$,  $\mc(m\alpha) \geq Cm^{n-c/n-k}$ for some constant $C$ and for sufficiently large $m$.  Furthermore, it is clear that that $\alpha$ is represented by irreducible cycles whose image in $Z$ has dimension $k-c$.  Thus $\alpha$ has all the desired properties.
\end{exmple}

\begin{rmk}
Note that the above construction is somewhat better than the ``naive'' lower bound given by complete intersections.  For example, consider again the class $\alpha = A^{2} + A \cdot H$ on $\mathbb{P}^{2} \times \mathbb{P}^{2}$ as in Example \ref{p2timesp2one}.  As demonstrated there, $\mc(m\alpha) \sim Cm^{3/2}$.   However, if we intersect members of $|m_{1}H|$ and $|m_{2}(H+A)|$ where $m_{1}m_{2}=m$, we can only obtain cycles through $Cm^{4/3}$ general points.
\end{rmk}

\section{Grassmannians}

$G(m,n)$ parametrizes $m$-dimensional subplanes of an $n$-dimensional complex vector space.  Fix a complete flag $V_{0} \subset \ldots \subset V_{n}$ in our vector space.  Given a non-increasing tuple of integers $\lambda = (\lambda_{1},\ldots,\lambda_{m})$ whose components $\lambda_{i}$ satisfy $0 \leq \lambda_{i} \leq n-m$, we let $\sigma_{\lambda}$ denote the class of the Schubert variety parametrizing linear subspaces $W$ satisfying for all $i$
\begin{equation*}
\dim(W \cap V_{n-m+i-\lambda_{i}}) \geq i.
\end{equation*}
As discussed in the introduction, there has been previous extensive work describing various forms of ``rigidity'' for Schubert classes on Grassmannians.  The study of the Iitaka dimension is a variation on this theme which yields interesting information for all classes.

\subsection{Iitaka dimensions on $G(2,n)$}

\begin{thrm} \label{iitakadimg2n}
The Iitaka dimension of a Schubert cycle on $G(2,n)$ is determined by the following list:
\begin{itemize}
\item $\kappa(\sigma_{1}) = \kappa(\sigma_{n-2,n-3}) = 2(n-2)$.
\item $\kappa(\sigma_{r}) = n-2$ for $1 < r \leq n-2$.
\item $\kappa(\sigma_{r,r-1}) = 2r$ for $1 < r < n-2$.
\item $\kappa(\sigma_{r,s}) = r+s$ otherwise.
\end{itemize}
\end{thrm}

\begin{rmk}
Note that Theorem \ref{iitakadimg2n} does not address the boundary classes that do not lie on extremal rays.  It would be interesting to see what behavior to expect along the rest of the pseudo-effective cone.
\end{rmk}

\begin{rmk}
The multi rigid classes are classified by \cite{rt12}, \cite{robles13} (see also \cite{bryant05}): on $G(2,n)$ multi rigid classes have the form $\sigma_{j,j}$ and $\sigma_{n-2,0}$.  These will automatically have the minimal Iitaka dimension, but note that the converse implication is not true.
\end{rmk}

\begin{rmk}
Certain features of this theorem should persist for all Grassmannians.  For example, consider $G(m,n)$ and suppose that $1 < t \leq m$.  Then we should have $\kappa(\sigma_{t}) = n-m$ and $\kappa(\sigma_{1^{t}}) = m$.
\end{rmk}

We prove each statement of Theorem \ref{iitakadimg2n} in turn.  Theorem \ref{iitakadimg2n}.(1) is obvious.

\subsubsection{The classes $\sigma_{r}$}

\begin{lem} \label{classificationofsigmar}
Consider $\sigma_{r}$ on $G(2,n)$ for some $1 < r \leq n-2$.  Then any irreducible cycle $V$ of class $m \sigma_{r}$ consists of the set of lines which intersect a fixed codimension $r+1$ degree $m$ subscheme in $\mathbb{P}^{n-1}$.  
\end{lem}

\begin{proof}
Take a general point $p$ on $\mathbb{P}^{n-1}$ and consider the cycle $Z_{p}$ representing the set of lines through $p$, so that $Z_{p}$ has class $\sigma_{n-2}$.  Using generality, the intersection of $Z_{p}$ and $V$ can be done on the cycle level (see \cite[2.Theorem]{kleiman74}).    For convenience we let $U \subset \mathbb{P}^{n-1}$ denote the open set of points  such that the intersection of $Z_{p}$ and $V$ can be done on a set-theoretic level.  The  cycle $Z_{p} \cdot V$ represents the class $m\sigma_{n-2,r}$ and by \cite[Theorem 5]{bryant05} consists of the lines through $p$ intersecting some codimension $r+1$ subscheme $Q_{p}$ of $\mathbb{P}^{n-1}$.
In this way we obtain a codimension $(r+1)$ subscheme $Q_{p}$ of $\mathbb{P}^{n-1}$ for each point $p \in U$.

Now we show that the $Q_{p}$ coincide as $p$ varies.  Take a general codimension $(n-r-1)$ plane $L$ in $\mathbb{P}^{n-1}$.  The locus $T$ parametrizing lines contained in $L$ has class $\sigma_{n-r-1,n-r-1}$.  Again we are in a setting of cycle-level intersection, and since the intersection must vanish we see that $T$ is disjoint from $V$.  Now consider varying $p$ through the points of $U \cap L$.  If the corresponding $Q_{p}$ varied in at least a one-dimensional family, then (after taking closures) we would find a line contained in $L$ represented by a point of $V$, a contradiction.  Thus $Q_{p}$ must be fixed as $p$ varies over points in $U \cap L$.  Since any pair of general points can be connected by a general codimension $(n-r-1)$ plane, this argument shows that $Q_{p}$ must be fixed as we vary $p$ over all general points of $\mathbb{P}^{n-1}$.  Taking a closure, we see that $V$ must be the set of lines intersecting a fixed codimension $(r+1)$ subscheme $Q$.

Finally, we must compare degrees.  If $Q$ has degree $d$, then the corresponding $V$ has class $m \sigma_{r}$ where
\begin{align*}
m & = V \cdot \sigma_{n-2,n-r} = V \cdot \sigma_{n-2} \cdot \sigma_{1}^{n-r} \\
& = \mathrm{cone}_{p}(Q) \cdot \sigma_{1}^{n-r} = d.
\end{align*}
\end{proof}

\begin{lem}
On  $G(2,n)$ we have $\kappa(\sigma_{r}) = n-2$ for any $1 < r \leq n-2$.
\end{lem}

\begin{proof}
We first show that this is a lower bound.  Fix a hyperplane $H$ in $\mathbb{P}^{n-1}$.  A general line in $\mathbb{P}^{n-1}$ will intersect $H$ at a general point.  So, any codimension $r$ subvariety of $H$ which contains $b$ general points will also intersect $b$ general lines as a codimension $r+1$ subvariety of $\mathbb{P}^{n-1}$.  A degree $m$ codimension $r$ subvariety $Z$ of $H$ can contain $\approx C m^{n-2/r}$ general points of $H$ for a positive constant $C$.  The set of lines intersecting $Z$ is a cycle on $G(2,n)$ of class $m\sigma_{r}$ going through $\approx C m^{n-2/r}$ general points.  (More precisely, by deforming $Z$ we obtain a family of cycles with the desired mobility count.)  This gives $\kappa(\sigma_{r}) \geq n-2$.

We next show that this is an upper bound.  Lemma \ref{classificationofsigmar} classifies all irreducible cycles whose class is proportional to $\sigma_{r}$.  In particular, an irreducible family of cycles on $G(2,n)$ representing $m\sigma_{r}$ will always be induced (at least over an open subset) by a 
family $p: U \to W$ of codimension $r+1$ degree $m$ subvarieties of $\mathbb{P}^{n-1}$.

We apply Lemma \ref{familymcbound} to $\mathbb{P}^{n-1}$ using for $q$ the family of lines on $\mathbb{P}^{n-1}$.  
We conclude that there is some constant $C$ so that a member of $p$ can meet at most $Cm^{n-2/r}$ general lines.  Thus, the corresponding family of cycles on $G(2,n)$ has mobility count at most $Cm^{n-2/r}$.
\end{proof}

\subsubsection{The classes $\sigma_{r,r-1}$}

We first recall a classical result of \cite{segre1}, \cite{segre2}.

\begin{lem}[\cite{segre2}] \label{classificationofsigmarr-1}
Consider $\sigma_{r,r-1}$ on $G(2,n)$ for some $1 < r < n-2$.  Then any irreducible cycle $V$ of class $m \sigma_{r,r-1}$ parametrizes either:
\begin{itemize}
\item the lines contained in the fibers of a one-dimensional family of $\mathbb{P}^{n-r-1}$s, or
\item the lines contained in a quadric hypersurface in some subplane of $\mathbb{P}^{n-1}$.
\end{itemize}
\end{lem}

To compute the mobility count of $\sigma_{r,r-1}$, it clearly suffices to focus on the first type of cycles.

\begin{lem}
On $G(2,n)$ we have $\kappa(\sigma_{r,r-1}) = 2r$ for any $1 < r < n-2$.
\end{lem}

\begin{proof}
We first rephrase the problem.  Note that the locus of $\mathbb{P}^{n-r-1}$s containing a fixed line in $\mathbb{P}^{n-1}$ is a Schubert variety of class $\sigma_{r,r}$ in $G(n-r,n)$.  Given the result of Lemma \ref{classificationofsigmarr-1}, it suffices to show that a curve in $G(n-r,n)$ of degree $m$ intersects at most $\approx C m^{\frac{2r}{2r-1}}$ Schubert varieties of class $\sigma_{r,r}$.

First we show the lower bound.  Fix a dimension $2r$ complete intersection variety $Y$ in $G(n-r,n)$.  Then a general Schubert variety of class $\sigma_{r,r}$ will intersect $Y$ in a finite number of points.  Since a degree $m$ curve  in $Y$ can contain $\approx C m^{2r/2r-1}$ general points of $Y$, we can also find a curve intersecting this many general Schubert varieties of class $\sigma_{r,r}$.

The upper bound follows from Lemma \ref{familymcbound}.
\end{proof}

\subsubsection{The other classes}

\begin{lem} \label{grassmanniancycledim}
Let $V \subset G(2,n)$ be an irreducible cycle with class proportional to $\sigma_{r,s}$ where $1 \leq s < r-1$.  Then the lines parametrized by $V$ sweep out an irreducible subset of $\mathbb{P}^{n-1}$ of codimension $s$.
\end{lem}

\begin{proof}
Clearly the lines sweep out an irreducible subset.  We have $\sigma_{r,s} \cdot \sigma_{n-s-1} = 0$ but $\sigma_{r,s} \cdot \sigma_{n-s-2} \neq 0$.  Noting that the Schubert cycle of type $\sigma_{k}$ parametrizes lines intersecting a fixed dimension $n-k-2$ linear subspace and using transversality of general intersections, we obtain the result.
\end{proof}

\begin{lem}
Consider $\sigma_{r,s}$ where $1 \leq s < r-1$.  Let $V$ be a cycle of class $m \sigma_{r,s}$ and let $Z$ denote the image in $\mathbb{P}^{n-1}$ of the universal family over $V$.  Suppose that $Z$ is irreducible.  Then either
\begin{itemize}
\item $Z$ is contained in a hyperplane, or
\item there is a unique fixed $n-2-r$ dimensional hyperplane $Q$ such that every line parametrized by $V$ intersects $Q$, and furthermore for every point $q$ of $Q$ there exists a  $\geq (n-2-s)$-dimensional subfamily of $V$ which parametrizes lines through $q$.
\end{itemize}
\end{lem}

\begin{proof}
The proof is by decreasing induction on $r$.  For $r = n-2$, this is proved by \cite[Theorem 5]{bryant05}.  (Note that both conditions $n-3 > s$ and $s \geq 1$ are necessary for the base case.)

We now consider the case $r < n-2$.  By Lemma \ref{grassmanniancycledim} $Z$ is irreducible of dimension $n-1-s$.

Choose a general hyperplane $H$ of $\mathbb{P}^{n-1}$ and the corresponding Schubert cycle $\sigma_{1,1}$.  By generality the intersection of $\sigma_{1,1}$ with $V$ can be done set theoretically to obtain a cycle $V'$.  Let $Z'$ denote the closed subset of $\mathbb{P}^{n-1}$ swept out by the lines parametrized by $V'$. Applying Lemma \ref{grassmanniancycledim} to components of $V'$ we see that each component of $Z'$ has dimension $n-s-2$.  Since $Z' \subset Z \cap H$ and this latter set is irreducible by Bertini, by dimension considerations we must have $Z' = Z \cap H$  and so $Z'$ is irreducible. 

By induction we see that either $Z'$ is contained in a hyperplane in $\mathbb{P}^{n-1}$ or every line parametrized by $V'$ intersects a fixed $n-3-r$ dimensional hyperplane $Q_{H}$ (necessarily contained in $H$).

In the first case, since $H$ is general we see that $Z$ is contained in a hyperplane of $\mathbb{P}^{n-1}$.

In the second case, consider the family of codimension $1$ hyperplanes $\widehat{H}$ in $\mathbb{P}^{n-1}$ containing $Q_{H}$.  As  the general such $\widehat{H}$ varies, we obtain a varying family of $Q_{\widehat{H}}$.  We claim that the $Q_{\widehat{H}}$ all coincide with $Q_{H}$.  Indeed, any line parametrized by $V'$ which also intersects some other point of $\widehat{H}$ is contained in $\widehat{H}$.  Since $s < r-1 < (n-2)-1$, the existence part of the inductive assumption yields through any point of $Q_{H}$ a $(n-4-s)$-dimension worth of lines contained in both $H$ and $\widehat{H}$.  Thus there is at least one line contained in $\widehat{H}$ through any point of $Q_{H}$, and by the uniqueness part of the inductive assumption, we must have $Q_{\widehat{H}} = Q_{H}$.

Finally note that as $H$ varies over general hyperplanes, the $Q_{H}$ are all hyperplane sections of a fixed $n-2-r$ dimensional plane $Q$.  Indeed, since $\sigma_{r,s} \cdot \sigma_{n-3-s,n-1-r} = 0$, we see that $V$ must not intersect a general Schubert variety of the latter class.  Such a variety parametrizes lines contained in a general $r$-dimensional hyperplane $L$ and intersecting an $s+1$-dimensional subplane $M$ of $L$.  If the $Q_{H}$ varied to cover a variety of dimension $n-1-r$, then some $Q_{H}$ would intersect $L$, and since through each point of $Q_{H}$ we have at least an $(n-3-s)$-dimensional family of lines there would be a line through that point also intersecting $M$, a contradiction.  Thus the union of all the $Q_{H}$ must have dimension $n-2-r$.  This union is obviously then a plane $Q$.

Finally we verify the two desired properties of $Q$ by induction.  Since $Q$ is the closure of the union of the $Q_{H}$, there is a line parametrized by $V$ through every point of $Q$.  If $Q$ were not unique, then a general hyperplane section would violate the inductive hypothesis.  Finally, an easy dimension count and induction argument verifies the existence part of the inductive assumption.
\end{proof}

For $r \geq 1$ there is an upper bound on the number of general lines which intersect any fixed $n-2-r$-dimensional hyperplane in $\mathbb{P}^{n-1}$.  Furthermore, the number of general lines contained in any degenerate subvariety is bounded above.  Thus we immediately obtain:

\begin{cor}
On $G(2,n)$ we have $\kappa(\sigma_{r,s}) = r+s$ whenever $1 \leq s < r-1$.
\end{cor}

\bibliographystyle{amsalpha}
\bibliography{iitakadim}

\newcommand{\etalchar}[1]{$^{#1}$}
\providecommand{\bysame}{\leavevmode\hbox to3em{\hrulefill}\thinspace}
\providecommand{\MR}{\relax\ifhmode\unskip\space\fi MR }
\providecommand{\MRhref}[2]{%
  \href{http://www.ams.org/mathscinet-getitem?mr=#1}{#2}
}
\providecommand{\href}[2]{#2}
\begin{thebibliography}{BCE{\etalchar{+}}02}

\bibitem[BCE{\etalchar{+}}02]{8authors}
Thomas Bauer, Fr{\'e}d{\'e}ric Campana, Thomas Eckl, Stefan Kebekus, Thomas
  Peternell, S{\l}awomir Rams, Tomasz Szemberg, and Lorenz Wotzlaw, \emph{A
  reduction map for nef line bundles}, Complex geometry ({G}\"ottingen, 2000),
  Springer, Berlin, 2002, pp.~27--36.

\bibitem[Bry05]{bryant05}
Robert Bryant, \emph{Rigidity and quasi-rigidity of extremal cycles
  in{H}ermitian symmetric spaces}, Princeton University Press Ann. Math.
  Studies AM-153, 2005.

\bibitem[Cam81]{campana81}
F.~Campana, \emph{Cor{\'e}duction alg{\'e}brique d'un espace analytique
  faiblement {K}{\"a}hl{\'e}rien compact}, Inv. Math. \textbf{63} (1981),
  no.~2, 187--223.

\bibitem[CC15]{cc15}
Dawei Chen and Izzet Coskun, \emph{Extremal higher codimension cycles on moduli
  spaces of curves}, Proc. Lond. Math. Soc. (3) \textbf{111} (2015), no.~1,
  181--204. \MR{3404780}

\bibitem[Cos11]{coskun11}
Izzet Coskun, \emph{Rigid and non-smoothable {S}chubert classes}, J.
  Differential Geom. \textbf{87} (2011), no.~3, 493--514.

\bibitem[Cos14]{coskun14}
\bysame, \emph{Rigidity of {S}chubert classes in orthogonal {G}rassmannians},
  Israel J. Math. \textbf{200} (2014), no.~1, 85--126.

\bibitem[CR13]{cr13}
Izzet Coskun and Colleen Robles, \emph{Flexibility of {S}chubert classes},
  Differential Geom. Appl. \textbf{31} (2013), no.~6, 759--774.

\bibitem[FKL16]{fkl15}
Mihai Fulger, J\'anos Koll\'ar, and Brian Lehmann, \emph{Volume and {H}ilbert
  function of {$\Bbb R$}-divisors}, Michigan Math. J. \textbf{65} (2016),
  no.~2, 371--387.

\bibitem[FL16]{fl15}
Mihai Fulger and Brian Lehmann, \emph{Morphisms and faces of pseudo-effective
  cones}, Proc. Lon. Math. Soc. \textbf{112} (2016), no.~4, 651--676.

\bibitem[FL17a]{fl14}
Mihai Fulger and Brian Lehmann, \emph{Positive cones of dual cycle classes},
  Algebr. Geom. \textbf{4} (2017), no.~1, 1--28.

\bibitem[FL17b]{fl13}
\bysame, \emph{Zariski decompositions of numerical cycle classes}, J. Algebraic
  Geom. \textbf{26} (2017), no.~1, 43--106.

\bibitem[Ful84]{fulton84}
William Fulton, \emph{Intersection theory}, Ergebnisse der Mathematik und ihrer
  Grenzgebiete (3) [Results in Mathematics and Related Areas (3)], vol.~2,
  Springer-Verlag, Berlin, 1984.

\bibitem[Hon05]{hong05}
Jaehyun Hong, \emph{Rigidity of singular {S}chubert varieties in {${\rm
  Gr}(m,n)$}}, J. Differential Geom. \textbf{71} (2005), no.~1, 1--22.

\bibitem[Hon07]{hong07}
\bysame, \emph{Rigidity of smooth {S}chubert varieties in {H}ermitian symmetric
  spaces}, Trans. Amer. Math. Soc. \textbf{359} (2007), no.~5, 2361--2381.

\bibitem[Iit70]{iitaka70}
Shigeru Iitaka, \emph{On {$D$}-dimensions of algebraic varieties}, Proc. Japan
  Acad. \textbf{46} (1970), 487--489.

\bibitem[Iit71]{iitaka71}
\bysame, \emph{On {$D$}-dimensions of algebraic varieties}, J. Math. Soc. Japan
  \textbf{23} (1971), 356--373.

\bibitem[Kle74]{kleiman74}
Steven~L. Kleiman, \emph{The transversality of a general translate}, Compos.
  Math. \textbf{28} (1974), no.~3, 287--297.

\bibitem[KMM92]{kmm92}
J{\'a}nos Koll{\'a}r, Yoichi Miyaoka, and Shigefumi Mori, \emph{Rationally
  connected varieties}, J. Algebraic Geom. \textbf{1} (1992), no.~3, 429--448.

\bibitem[Laz04]{lazarsfeld04}
Robert Lazarsfeld, \emph{Positivity in algebraic geometry. {I}-{II}},
  Ergebnisse der Mathematik und ihrer Grenzgebiete. 3. Folge. A Series of
  Modern Surveys in Mathematics [Results in Mathematics and Related Areas. 3rd
  Series. A Series of Modern Surveys in Mathematics], vol.~48, Springer-Verlag,
  Berlin, 2004, Classical setting: line bundles and linear series.

\bibitem[Leh16]{lehmann16}
Brian Lehmann, \emph{Volume-type functions for numerical cycle classes}, Duke
  Math. J. \textbf{165} (2016), no.~16, 3147--3187. \MR{3566200}

\bibitem[Rob13]{robles13}
C.~Robles, \emph{Schur flexibility of cominuscule {S}chubert varieties}, Comm.
  Anal. Geom. \textbf{21} (2013), no.~5, 979--1013.

\bibitem[RT12]{rt12}
C.~Robles and D.~The, \emph{Rigid {S}chubert varieties in compact {H}ermitian
  symmetric spaces}, Selecta Math. (N.S.) \textbf{18} (2012), no.~3, 717--777.

\bibitem[Seg48a]{segre1}
Beniamino Segre, \emph{Sulle {$V_n$} contenenti pi\`u di {$\infty^{n-k}S_k$}.
  {I}}, Atti Accad. Naz. Lincei. Rend. Cl. Sci. Fis. Mat. Nat. (8) \textbf{5}
  (1948), 193--197.

\bibitem[Seg48b]{segre2}
\bysame, \emph{Sulle {$V_n$} contenenti pi\`u di {$\infty^{n-k}S_k$}. {II}},
  Atti Accad. Naz. Lincei. Rend. Cl. Sci. Fis. Mat. Nat. (8) \textbf{5} (1948),
  275--280.

\bibitem[Voi10]{voisin10}
Claire Voisin, \emph{Coniveau 2 complete intersections and effective cones},
  Geom. Funct. Anal. \textbf{19} (2010), no.~5, 1494--1513.

\bibitem[Wal97]{walters97}
Maria~F. Walters, \emph{Geometry and uniqueness of some extreme subvarieties in
  complex {G}rassmannians}, ProQuest LLC, Ann Arbor, MI, 1997, Thesis
  (Ph.D.)--University of Michigan.

\end{thebibliography}

\end{document}